\documentclass{amsart}

\usepackage{amssymb}

\newtheorem{thm}{Theorem}
\newtheorem{lem}{Lemma}
\newtheorem{prop}{Proposition}
\newtheorem{rem}{Remark}

\theoremstyle{definition}
\newtheorem{example}{Example}

\def\Hom{\mathop{\fam 0 Hom}\nolimits}
\def\Cur{\mathop{\fam 0 Cur}\nolimits}
\def\Vir{\mathop{\fam 0 Vir}\nolimits}
\def\idd{\mathop{\fam 0 id}\nolimits}

\title[On the Ado Theorem for
  finite Lie conformal algebras]{On the Ado Theorem for
  finite Lie conformal algebras with Levi decomposition}
\author{Pavel Kolesnikov}

\thanks{Partially supported by RFBR (project 12.01.00329), FAPESP 2012/04704-0, and the Federal Target Program (contract 14.740.11.0346). }

\address{Sobolev Institute of Mathematics, Novosibirsk, Russia}

\begin{document}

\begin{abstract}
We prove that a finite torsion-free conformal Lie algebra
with a splitting solvable radical has a
finite faithful conformal representation.
\end{abstract}

\maketitle

\section{introduction}
Conformal algebras were introduced in \cite{Kac1996} as an algebraic
 formalism
describing properties of the singular part of the operator product
expansion (OPE) in conformal field theory (i.e., conformal Lie
algebra is a singular part of a vertex operator algebra). The
same notion is often called vertex Lie algebra (see, e.g.,
\cite{FBZ2001}).
The relations of conformal algebra with other areas of mathematics
include hamiltonian formalism in the theory of evolution equations
\cite{BDK2001}, variational calculus \cite{DSK2008}, and the
theory of algebraic structures arising from algebraic topology and
K-theory \cite{Loday2001}, see \cite{Kol2008}.

From the formal point of view,
the category of conformal algebras is the first member
(after ``ordinary'' algebras over a field) in
the hierarchy of pseudo-algebras over cocommutative bialgebras, as described in \cite{BDK2001}.
The case of ordinary algebras corresponds to the one-dimensional bialgebra $H=\Bbbk $,
conformal algebras are modules over the polynomial bialgebra $H=\Bbbk [\partial ]$.
As in the case of ordinary
algebras, an associative conformal algebra turns into a Lie conformal algebra
by means of the standard morphism of
the corresponding operads given by $x_1x_2\mapsto x_1x_2-x_2x_1$.
The converse statement does not hold in general: There exist Lie conformal algebras
that cannot be embedded into an associative one \cite{Roit2000}.
However, it is an open problem whether a {\em finite\/} (i.e.,
finitely generated over $H$) conformal Lie algebra
has such an embedding. A more precise problem can be stated as
follows: Whether a finite conformal Lie algebra which is torsion-free
as an $H$-module can be faithfully represented by
{\em conformal endomorphisms} \cite{Kac1996}
of a finitely generated torsion-free $H$-module $M$?
This statement would be the ``conformal analogue'' of the
classical Ado Theorem. The purpose of this paper is
to make a step towards the solution of this problem.

The Ado Theorem is a fundamental fact in the theory of
Lie algebras, however, it has a reputation of ``strange'' theorem
\cite{Ner2003}. Every known proof of this statement (e.g., the
classical ones \cite{JacLie}, \cite{Proc2005},
 or the short and elegant proof
in \cite{Ner2003}) exploits the following basic properties of
finite-dimensional Lie algebras:
\begin{itemize}
\item[(A1)] The Poincar\'e---Birkhoff---Witt (PBW) Theorem (at least
for nilpotent algebras);
\item[(A2)] Complete reducibility of finite-dimensional modules
over semisimple algebras;
\item[(A3)] The image of a solvable Lie algebra under its derivation
is nilpotent;
\item[(A4)] The Levi Theorem (splitting of the solvable radical).
\end{itemize}
For conformal algebras (even for finite ones),
these properties do not hold in general.
This is the reason why proving an analogue of the
Ado Theorem for conformal algebras is a challenging problem.

In \cite{Roit2005}, it was shown that a nilpotent
conformal Lie algebra embeds into a nilpotent associative
conformal algebra, thus, the property (A1) is not essential
for our purpose.

In \cite{Kol2011}, we eluded the property (A2):
The existence of a finite faithful representation
was proved for torsion-free finite Lie conformal algebras
 of ``classical type'', i.e., with a splitting solvable radical
and without Virasoro elements in the semisimple part.
These conformal algebras turn to be subalgebras of
current conformal algebras over finite-dimensional
ordinary Lie algebras. It is interesting to note
that the proof in \cite{Kol2011} was not based on
the Ado Theorem for ordinary algebras.
Therefore, the problem was solved for those conformal Lie
algebras satisfying the analogues of (A3) and (A4).

The purpose of this note is to eliminate (A3).
Namely, we prove the conformal version
of the Ado Theorem for an arbitrary finite torsion-free
conformal Lie algebra $L$
with a splitting solvable radical.

Let $\Bbbk $ be a field of zero characteristic.
Without loss of generality (see \cite[Lemma 7]{Kol2011})
we may assume $\Bbbk $
to be algebraically closed.
Throughout the paper, the symbol $\otimes $ without a subscript means
the tensor product of spaces over $\Bbbk $.

\section{Conformal algebras and their representations}

A {\em conformal algebra} \cite{Kac1996} is a unital module $C$ over
the polynomial algebra $H=\Bbbk[\partial ]$  equipped with a polynomial-valued
$\Bbbk $-linear operation
$(\cdot _{(\lambda ) \cdot}): C\otimes C \to C[\lambda ]$,
where $\lambda $ is a formal variable (this operation is called
{\em $\lambda $-product}),
satisfying the following axioms:
\begin{equation}\label{eq:SesquiLinear}
 (\partial a_{(\lambda )} b) = -\lambda (a_{(\lambda )} b) ,\quad
(a_{(\lambda )} \partial b) =(\lambda +\partial)(a_{(\lambda )} b).
\end{equation}

Conformal algebra is said to be {\em finite\/} if
it is finitely generated as a module over~$H$.

Equivalently, one may consider a conformal algebra $C$
as a linear space
over $\Bbbk $
with one linear operation $\partial $ and with an infinite family
of bilinear products $(\cdot _{(\alpha )}\cdot )$, $\alpha \in \Bbbk $,
where
$(a_{(\alpha )} b) = (a_{(\lambda )} b)|_{\lambda =\alpha}$
for $a,b\in C$.

Another interpretation of a conformal algebra structure is based on
the notion of a pseudo-tensor category.
Let us sketch the ideas developed in \cite{BDK2001} to clarify
the relations between ``ordinary'' and conformal algebras and
their representations.

Consider the polynomial algebra $H=\Bbbk[\partial ]$
and denote by $\mathcal M(H)$ the class of (left)
unital $H$-modules.
Recall that $H$ carries the standard
bialgebra structure given by
coproduct $\Delta (\partial) = \partial \otimes 1 + 1\otimes \partial $
and counit $\varepsilon(\partial ) = 0$.
Given $M_1, \dots, M_n, M\in \mathcal M$, denote
\begin{equation}\label{eq:HomP}
 P_n(M_1,\dots, M_n; M)= \Hom_{H^{\otimes n}} (M_1\otimes \dots \otimes M_n, H^{\otimes n}\otimes _H M),
\end{equation}
where $H^{\otimes n}$ is considered as the outer product
of regular right $H$-modules.
For the class $\mathcal M(H)$, the spaces \eqref{eq:HomP}
play the same role as the spaces of
polylinear maps for the class of linear spaces over a field.
There exist a composition rule
and an equivariant symmetric group action on \eqref{eq:HomP}
such that $\mathcal M(H)$ turns into a pseudo-tensor category in the
sense of  \cite{BD2004}.
Conformal algebra $C$ is exactly the same as an algebra
in $\mathcal M(H)$
with one binary operation (pseudo-product) $\mu \in P_2(C,C; C)$.
Namely, for $a,b\in C$,
\[
 (a_{(\lambda)} b) = \sum\limits_{n\ge 0} \lambda ^{n}
   c_n, \quad c_n\in C,
\]
we have
\[
 \mu (a,b) = \sum\limits_{n\ge 0} ((-\partial )^n \otimes 1) \otimes _H c_n.
\]
In these terms, \eqref{eq:SesquiLinear}
is equivalent to $H^{\otimes 2}$-linearity of~$\mu $.

These data are enough to define what is an algebra
(associative, commutative, Lie, etc.), a homomorphism,
a representation (module),
and a cohomology in the class $\mathcal M(H)$ \cite{BDK2001}.

For example, the associativity of a conformal
algebra $C$ may be expressed in terms of pseudo-product
as
\[
 \mu (\mu, \idd) = \mu(\idd, \mu) \in P_3(C,C,C; C),
\]
and in terms of $\lambda $-product as
\[
 (a_{(\alpha )} (b_{(\beta )} c)) = ((a_{(\alpha )} b)_{(\alpha+\beta )} c), \quad a,b,c\in C,\ \alpha,\beta\in \Bbbk .
\]
Similarly, the anti-commutativity and the Jacobi
identity for a pseudo-product $\mu $ have the same form
as for ``ordinary" product:
\[
 \begin{gathered}
  \mu = -\mu^{(12)} , \\
  \mu (\idd, \mu ) - \mu (\idd, \mu )^{(12)} =   \mu (\mu, \idd).
 \end{gathered}
\]
In terms of $\lambda $-product, these relations turn into
\[
 \begin{gathered}
  (a_{(\alpha )} b) = - (b_{(-\partial - \alpha )} a), \\
 (a_{(\alpha )} (b_{(\beta )} c)) - (b_{(\beta )} (a_{(\alpha )}  c))= ((a_{(\alpha )} b)_{(\alpha+\beta )} c) ,
 \end{gathered}
\]
for $a,b,c\in C$, $\alpha,\beta \in \Bbbk $,
respectively.

As in the case of ordinary algebras, an associative conformal
algebra $C$ with respect to
the new $\lambda $-product
\[
 [a_{(\alpha )} b] = (a_{(\alpha )} b) - (b_{(-\partial - \alpha )} a)
\]
satisfies anti-commutativity and Jacobi identity,
i.e., is a Lie conformal algebra.
Below we will use the brackets $[\cdot_{(\lambda )}\cdot ]$
to denote $\lambda $-products in all Lie conformal algebras.

The notions of an ideal, solvability, and nilpotence
have the natural interpretation in conformal
 algebras.
A conformal algebra $C$ is called {\em trivial\/} if
$C_{(\alpha )} C = 0$ for all $\alpha \in \Bbbk $.
A nontrivial conformal algebra is {\em simple\/} if it has
no nonzero proper ideals. If $C$ has no nonzero
solvable ideals then it is said to be {\em semisimple}.

To define a conformal algebra structure on an $H$-module, it is enough
to define the $\lambda $-product on its generators over $H$
and then expand it to
 the entire module by means of \eqref{eq:SesquiLinear}.

\begin{example}
Let $\mathfrak a$ be an ordinary algebra over $\Bbbk $.
Then the free $H$-module $H\otimes \mathfrak a$
can be considered as a conformal algebra with respect to
operation
\[
a_{(\lambda )} b = ab,
\quad a,b\in \mathfrak a,
\]
called {\em current conformal algebra} $\Cur \mathfrak a$.
If $\mathfrak a$ is associative (Lie, alternative, etc) then so is
$\Cur \mathfrak a$.
\end{example}

\begin{example}
The free $H$-module of rank one generated by an element $v$
turns into a Lie conformal algebra by means of the operation
\[
[v_{(\lambda )} v] = (\partial +2\lambda )v.
\]
This structure is called {\em Virasoro conformal algebra} $\Vir $.
\end{example}

\begin{example}
Suppose $\mathfrak g$ is a Lie algebra.
The direct sum of $H$-modules
$Hv\oplus (H\otimes \mathfrak g)$
with respect to
\[
\begin{gathered}[]
[v_{(\lambda )} v] = (\partial +2\lambda )v,
\quad
[v_{(\lambda )} a] = (\partial+\lambda )a, \ a\in \mathfrak g\\
[a_{(\lambda )} b] = [a,b],
\ a,b\in \mathfrak g,
\end{gathered}
\]
 is a Lie conformal algebra denoted by
$\Vir\ltimes \Cur \mathfrak g$,
the semi-direct product of conformal algebras.
\end{example}

The classification of finite conformal algebras is given by the following

\begin{thm}[\cite{DK1998}]\label{thm:ConfAlgebras}
{\rm (i)} A simple finite Lie conformal algebra
is isomorphic either to $\Vir $ or to $\Cur \mathfrak g$,
where $\mathfrak g$ is a finite-dimensional simple Lie algebra.

{\rm (ii)} A semi-simple finite Lie conformal algebra is
a direct sum of conformal algebras $C_1\oplus \dots \oplus C_k$,
where each $C_i$ is either simple or isomorphic to
$\Vir\ltimes \Cur\mathfrak g$, $\mathfrak g$ is a simple finite-dimensional
Lie algebra.
\end{thm}

An arbitrary finite associative or Lie conformal algebra $C$
obviously has a maximal solvable ideal (radical) $R$ such that $C/R$ is
a semisimple conformal algebra. However, for Lie conformal algebras
there is no analogue of the Levi Theorem stating
$C \simeq (C/R)\ltimes R$.
As in the case of ordinary Lie algebras, the abelian extensions of
conformal algebras are described in terms of the second
cohomology group. The corresponding notions were introduced
in \cite{BKV1999}, were cohomologies of
simple and semisimple finite conformal algebras with coefficients
in their irreducible modules were computed;
many of them are nontrivial.

Suppose $C $ is an associative (Lie) conformal algebra.
Then a {\em conformal module\/} over $C$ is an $H$-module
$M\in \mathcal M(H)$ equipped with
$\nu \in P_2(C,M;M)$ satisfying the appropriate
associativity (Jacobi) identity. It terms of
$\lambda $-product these notions were introduced
and studied in \cite{CK1997}

\begin{rem}
In this study, we consider those conformal algebras
(and their modules)
that are torsion-free as $H$-modules. The reason
for such a restriction comes from the following
observation \cite{DK1998, BDK2001}. If $U,V,W\in \mathcal M(H)$,
$\mu \in P_2(U,V;W)$, then
$\mu (a,V)=\mu(U,b)=0$ for every torsion elements $a\in U$,
$b\in V$. Hence, a conformal algebra with a nonzero $H$-torsion
has no faithful representations.
\end{rem}

Irreducible representations of finite simple and
semisimple Lie conformal algebras are
described by

\begin{thm}[\cite{CK1997}]\label{thm:ConfModules}
{\rm (i)} A finite irreducible conformal module over
$\Vir $ is a free $H$-module of rank one generated by
an element $u$ such that
\[
 v_{(\lambda )} u = (\alpha +\partial +\Delta\lambda )u, \quad \alpha,\Delta\in \Bbbk, \ \Delta\ne 0.
\]
Such a module is denoted $M_{\alpha, \Delta}$.

{\rm (ii)} Suppose $\mathfrak g$ is a finite-dimensional
simple Lie algebra. A finite irreducible conformal module over $\Cur\mathfrak{g}$
is isomorphic to $H\otimes U$, where $U$ is a finite-dimensional irreducible
$\mathfrak{g}$-module, and
\[
 g_{(\lambda)} u = gu, \quad g\in \mathfrak{g}, \ u\in U.
\]
This module is natural to denote by $\Cur U$.

{\rm (iii)}
Suppose $\mathfrak g$ is a finite-dimensional
simple Lie algebra.
A finite conformal module $M_{\alpha,\Delta,U}$ over $\Vir\ltimes \Cur\mathfrak{g}$
is constructed as $H\otimes U$, where
$U$ is a finite-dimensional $\mathfrak{g}$-module, and
\[
\begin{gathered}
v_{(\lambda )} u  = (\alpha + \partial +\Delta\lambda )u,
\quad \alpha,\Delta\in \Bbbk, \\
g_{(\lambda )} u =gu, \quad g\in \mathfrak{g}, \ u\in U.
\end{gathered}
\]
Every finite irreducible conformal module over $\Vir\ltimes \Cur\mathfrak{g}$
is isomorphic to $M_{\alpha,\Delta,U}$,
where either $U$ is an irreducible $\mathfrak{g}$-module (and $\Delta $ is an
arbitrary scalar) or
$U$ is trivial one-dimensional and $\Delta \ne 0$.
\end{thm}

If $L$ is a Lie conformal algebra and $a\in L$ then
the operation $D_\lambda =[a_{(\lambda )}\cdot ]: L\to L[\lambda ]$
has the following property:
$D_\lambda ([x_{(\mu)} y]) = [(D_\lambda x)_{(\lambda +\mu )} y] +
 [x_{(\mu)} (D_\lambda y)]$
for all $x,y \in L$.
In general, such a map $D_\lambda $ is called a {\em conformal derivation} of $L$.

\begin{lem}[c.f. {\cite[Proposition 3.1]{CK1997}}]\label{lem:Triv*Irred=0}
Assume $C$ is a conformal algebra which is a conformal
module over $\Vir =Hv$ such that $D_\lambda = (v_{\lambda } \cdot )$
is a conformal derivation of $C$.
Suppose $M$ is a conformal $\Vir $-submodule of $C$ isomorphic to $M_{\alpha,\Delta}$
for some $\alpha,\Delta \in \Bbbk $, and $N$ is a trivial conformal $\Vir $-submodule
of $C$. Then $(N_{(\lambda )} M) =0$ in $C$.
\end{lem}

\begin{proof}
 Denote $M=Hu$, where
 $(v_{(\lambda )}u )= (\alpha+\partial +\lambda \Delta)u$.
Choose an arbitrary $a\in N$  and consider
$a_{(\lambda )}u = \sum_i \lambda^ i c_i \in C[\lambda ]$.
Then
\begin{equation}\label{eq:DerProd1}
D_0 (a_{(\lambda)} u) = a_{(\lambda)} D_0u = (\alpha +\lambda +\partial ) (a_{\lambda } u).
\end{equation}
Since $D_0 \big(\sum_i \lambda^i c_i \big) = \sum_i \lambda ^i D_0c_i$,
the degree in $\lambda $ of the left-hand side of \eqref{eq:DerProd1}
does not exceed the degree of $(a_{(\lambda )} u)$. If the latter is finite
then the right-hand side of \eqref{eq:DerProd1} has a greater degree
than  $(a_{(\lambda )} u)$. The contradiction obtained proves $(a_{(\lambda )} u)=0$.
\end{proof}

\section{Composition series in conformal modules}

Note that a finite conformal module even over a finite
simple Lie conformal algebra cannot (in general) be decomposed
into a direct sum of irreducible ones
(see \cite{CKW1998} for a systematic study of extensions).
Moreover, although the lattice of conformal submodules
in a given module
is modular (Dedekind), a finite composition series may not exist
even in a finite conformal module. As an example, consider
the free $H$-module $M_1$ of rank one with respect to a trivial
action of a conformal algebra (say, over $\Vir $).
Then there exists a normal series
\[
  0\subset \partial^n M_1 \subset \partial^{n-1}M_1\subset \dots \subset \partial M\subset M
\]
of arbitrary length $n$. However, we may still apply a kind of
triangular decomposition to those conformal modules we need (see
Lemma~\ref{lem:Triang} below).

Some of the results of this section can be recovered from
\cite{CK1997, CKW1998}, but we state
their proofs for readers' convenience.

\begin{lem}\label{lem:SplitExtMod}
Let $L$ be a conformal algebra of type $\Vir\ltimes \Cur\mathfrak g$,
where $\mathfrak g$ is either a finite-dimensional simple
Lie algebra or $\mathfrak g=0$ (i.e., $L=\Vir $).
Suppose $V$, $M$, and $E$ are three finite conformal modules over $L$
such that $M$ is irreducible, $V$ is a trivial torsion-free
$L$-module, and there exists a short exact sequence
\begin{equation}\label{eq:ExtSequence}
0\to V\to E\to M\to 0
\end{equation}
of conformal modules over $L$. Then $E\simeq V\oplus M$, the direct sum
of conformal modules over $L$.
\end{lem}

\begin{proof}
Extensions of conformal modules may be described via the corresponding
conformal cocycles (see \cite{CKW1998, DSK2008}).
Let $L$, $M$, and $V$ be as in the statement.
A $\Bbbk $-linear map
$\varphi_\lambda : L\otimes M \to V[\lambda ]$ satisfying the 3/2-linearity condition
similar to \eqref{eq:SesquiLinear} is called a cochain.
If a cochain $\varphi_\lambda $ satisfies
\begin{equation}\label{eq:Cocycle}
 \varphi_\lambda (a, b_{(\mu )} u) - \varphi_\mu (b, a_{(\lambda )} u) = \varphi_{\lambda +\mu}([a_{(\lambda)} b], u)
\end{equation}
for all $a,b\in L$, $u\in M$ then
$\varphi_\lambda $ is said to be a cocycle.
For every $H$-linear map $\tau : M\to V$ its differential
$\delta_\lambda\tau : L\otimes M \to V$
defined by $(\delta_\lambda\tau)(a, u) = \tau(a_{(\lambda )}u)$
is a cocycle.

Obviously, every cocycle $\varphi_\lambda : L\otimes M \to V[\lambda ]$
gives rise to a conformal $L$-module $E=E(M,V,\varphi )$
such that the sequence \eqref{eq:ExtSequence} is exact.
Conversely, every exact sequence \eqref{eq:ExtSequence}
allows to define the corresponding cocycle \cite[Theorem~2.1]{DSK2008}.
Moreover, for given cocycles $\varphi_\lambda $ and $\psi_\lambda $,
the extensions $E(M,V,\varphi )$ and
$E(M,V,\psi )$ are isomorphic if and only if
$\varphi_\lambda - \psi_\lambda = \delta_\lambda \tau$
for an appropriate~$\tau : M\to V$.

To prove the statement, it is enough to show
that every cocycle $\varphi_\lambda: L\otimes M\to V$
is equal to $\delta_\lambda \tau $ for some $\tau: M\to V$.
By Theorem \ref{thm:ConfModules}, there are three cases to be considered.

Case 1: $L=\Vir$, $M=M_{\alpha, \Delta}$, $\Delta\ne 0$;

Case 2: $L=\Vir\ltimes \Cur \mathfrak g$, $M=M_{\alpha, \Delta, U}$, $\Delta\ne 0$;

Case 3:  $L=\Vir\ltimes \Cur \mathfrak g$, $M=M_{\alpha, 0, U}$, $U$ is a nontrivial $\mathfrak g$-module.

Let us consider Case 2 and Case 3 in details since Case 1 is completely covered by calculations
from Case 2.

In Case 2, suppose $\varphi_\lambda (v,u) = \sum_i f_i(\partial, \lambda )w_i(u)$, $u\in U$, where
$f_i(\partial, \lambda )\in \Bbbk [\partial, \lambda ]$, $w_i(u)$ are linearly
independent over $\Bbbk $ in $V$.
Then by \eqref{eq:Cocycle} we have (for $a=b=v$, $\mu = 0$)
\[
 (\alpha +\partial) f_i(\partial, \lambda ) = (\alpha +\Delta\lambda + \partial) f_i(\partial, 0).
\]
Since $\Delta\ne 0$, $\alpha +\partial$
divides
$f_i(\partial, 0)$, so $f_i(\partial, \lambda ) = (\alpha+\Delta\lambda +\partial )h_i(\partial )$.
Define
$\tau(u) = \sum_i h_i(\partial )w_i(u) \in V$.
Then
$\varphi_\lambda (v,u) = \tau(v_{(\lambda )} u) = (\delta_\lambda \tau)(v,u)$.
Therefore, without loss of generality (replacing $\varphi_\lambda  $ with
$\varphi_\lambda - \delta_\lambda \tau $) we may assume
$\varphi_\lambda (v, u) = 0$.

Now, consider $g\in \mathfrak g$.
Then by \eqref{eq:Cocycle} we have (for $a=v$, $b=g$, $\lambda = 0$)
\[
\varphi_0(v, gu) - \varphi_\mu(g, (\partial+\alpha )u) = \varphi_\mu(\partial g, u),
\]
that implies $(\alpha+\partial)\varphi_\mu (g,u)=0$. Since $V$
is torsion-free, we have $\varphi_\lambda \equiv 0$.

In Case 3, the same computations with $a=b=v$ imply that
$\varphi_\lambda (v, u)$ do not depend on $\lambda $ for every $u\in U$.
Suppose $V=H\otimes W$, $\{w_i\}_{i\in I }$ is a basis of
the linear space $W$. Then for every $u\in U$ we may write
$\varphi_\lambda (v,u) = \sum_i f_i^u(\partial )w_i$,
$f_i^u \in H$.
Let us represent
$f_i^u(\partial )$ as $(\alpha +\partial )h_i^u(\partial ) + \beta_i^u$, $\beta_i^u\in \Bbbk $.
Then for $\tau(u) = \sum_i h_i^u(\partial )w_i$
we have $(\delta_\lambda\tau)(v,u) = \varphi_\lambda (v,u) - \sum_i \beta_i^u w_i$.
Hence, we may assume without loss of generality that
$f_i^u(\partial )$ are constants from $\Bbbk $.

Suppose
$\varphi_\lambda (g,u) = \sum_i h^{g,u}_i(\partial, \lambda )w_i$,
$h^{g,u}_i \in \Bbbk [\partial, \lambda ]$.
By the same reasons as in  Case 2 we have
\[
 \varphi_0(v, gu) = (\alpha +\partial )\varphi_\mu (g,u),
\]
so $f_i^{gu} = (\alpha +\partial ) h_i^{g,u}(\partial, \mu)$.
Since the left-hand side is constant, we obtain
$f_i^{gu} = h_i^{g,u} = 0$ for all $g\in \mathfrak g$, $u\in U$.
Hence, $\varphi_\lambda \equiv 0$.
\end{proof}

\begin{rem}
 Note that Lemma~\ref{lem:SplitExtMod} does not hold
for $L=\Cur \mathfrak{g}$ for simple finite-dimensional
Lie algebras $\mathfrak g$, see \cite[Proposition 4.4]{CKW1998}.
\end{rem}

\begin{lem}\label{lem:ExistIrreducible}
Let $L$ be as in Lemma~\ref{lem:SplitExtMod}, and let
$V$ be a finite non-trivial torsion-free conformal
module over $L$. Then $V$ contains an irreducible
conformal submodule.
\end{lem}

\begin{proof}
 Choose a nontrivial conformal $L$-submodule $W$ of minimal rank in $V$.
Denote by $\mathcal R$ the set of all those nontrivial conformal $L$-submodules
in $W$ that have the same rank over $H$ as $W$.
For every $U\in \mathcal R$ there exists $h\in H$, $h\ne 0$, such that
$hW\subseteq U$ (since $W/U$ coincides with its torsion).
Denote
\[
 V_0 = \bigcap\limits_{U\in \mathcal R} U,
\]
then $L_{(\lambda )} W \subseteq V_0[\lambda ]$.
Since $L$ is a perfect algebra ($\sum_{\alpha \in \Bbbk }[L_{(\alpha )} L]=L$),
the last expression implies that $V_0$ is a nontrivial
conformal $L$-module. Hence, the rank of $V_0$ coincides with the rank of $W$,
and there are no nontrivial proper conformal $L$-submodules in~$V_0$.

Choose a maximal (proper) conformal $L$-submodule $U_0$ in $V_0$.
By the construction of $V_0$, $U_0$ has to be trivial.
The maximality of $U_0$ implies $M_0=V_0/U_0$ to be an  irreducible
conformal $L$-module.
By Lemma \ref{lem:SplitExtMod} there exists a conformal $L$-submodule
 in $V_0$ isomorphic to $M_0$. This is the desired submodule.
\end{proof}

\begin{lem}\label{lem:Triang}
Let $L$ be as in Lemma~\ref{lem:SplitExtMod},
and let $M$ be a finite conformal module over $L$.
Then there exists a chain of submodules
\[
0=M_{-1}\subset M_0\subset \dots \subset M_n = M,
\]
where $M_k/M_{k-1}$ is either irreducible or
trivial torsion-free or coincides with its torsion
(hence, trivial).
\end{lem}

\begin{proof}
This is an immediate corollary of Lemma \ref{lem:ExistIrreducible}.
Note that if $M$ is nontrivial torsion-free conformal $L$-module
then $M_0$ has to be irreducible.
\end{proof}

\begin{lem}\label{lem:Any*Irred=Triv}
Assume $C$ is a torsion-free conformal algebra which is a conformal
module over $\Vir =Hv$ such that $D_\lambda = (v_{\lambda } \cdot )$
is a conformal derivation of $C$.
Suppose $M$ is a conformal $\Vir $-submodule of $C$ isomorphic to $M_{\alpha,\Delta}$
for some $\alpha,\Delta \in \Bbbk $, and $N$ is an arbitrary finite conformal $\Vir $-submodule
of $C$. If $(N_{(\lambda )} M)$ falls into the kernel of $D_0$
then $(N_{(\lambda )} M) =0$ in $C$.
\end{lem}

\begin{proof}
Let $M=Hu$, $D_0u = (\alpha +\partial )u$.
Consider the triangular decomposition from Lemma \ref{lem:Triang} for $N$:
\[
 0=N_{-1}\subset N_0\subset \dots \subset N_r = N.
\]
Assume $k\ge 0$ is the minimal index such that $(N_k{}_{(\lambda )} M)\ne 0$.

If $N_k/N_{k-1}$ is a trivial $\Vir$-module then
for every $a\in N_k$ we have
$D_0a \in N_{k-1}$, so
\[
 0 = D_0(a_{(\lambda )}u) = a_{(\lambda )} D_0u = (\alpha+\partial +\lambda )(a_{(\lambda )} u).
\]
Since $C$ is torsion-free, $(a_{(\lambda )} u)=0$.

If $N_k/N_{k-1}$ is an irreducible $\Vir$-module isomorphic to $M_{\beta, \delta}$
then $N_k = Ha + N_{k-1}$, where $D_0a \in (\beta +\partial )a + N_{k-1}$.
In this case,
\[
 0=D_0(a_{(\lambda )}u) = (D_0a_{(\lambda )} u) + a_{(\lambda)} D_0u = (\beta+\alpha +\partial)(a_{(\lambda )}u),
\]
so $a_{(\lambda )}u =0$.
\end{proof}

\section{Finite faithful representation}

Suppose $L$ is a finite torsion-free conformal Lie algebra
with the maximal solvable ideal $R$.
Assume $R$ splits in $L$, i.e., $L=L_0\ltimes R$, where $L_0$ is semisimple.
Denote by $Z(L)$ the center of $L$.

First, consider the case when $R$ is nilpotent.

\begin{prop}\label{prop:IdealExists}
If $R$ is nilpotent and $L_0$ contains a summand $L_1$
isomorphic either to $\Vir $ or
to $\Vir\ltimes \Cur\mathfrak g$
($\mathfrak g$ is a simple finite-dimensional Lie algebra)
such that $R$ is a nontrivial conformal $L_1$-module
then $L$ contains a nonzero ideal $I$ such that $I\cap Z(L)=0$
and $[R_{(\lambda )} I] =0$.
\end{prop}

\begin{proof}
Consider the sequence of ideals
\[
 R=R^1\supset R^2\supset \dots \supset R^{n-1}\supset R^n=0,
\]
where
$ R^{l+1} = \sum\limits_{\alpha \in \Bbbk }  [ R_{(\alpha)} R^l ] $.
Choose the maximal $m$ such that $R^m$ is a nontrivial
$L_1$-module.
By Lemma \ref{lem:ExistIrreducible} there exists an irreducible
conformal $L_1$-submodule $M_0$ in $R^m$.

It turns out that $M_0$ is the desired ideal in $L$.
Let us consider in details the case when
$L_1=\Vir\ltimes \Cur \mathfrak g$
and $M_0=M_{\alpha,\Delta, U}$
as in Theorem \ref{thm:ConfModules}(iii).
By $v$ we denote the canonical Virasoro element of $L_1$.
It is enough to show
that $[L_k {}_{(\lambda )} M_0]\subseteq M_0[\lambda ]$
for all $k>1$ and $[R_{(\lambda )}M_0]\subseteq M_0[\lambda ]$.

Let $u\in U$. Then
$Hu$ is a conformal $\Vir$-submodule in $L$
which is isomorphic to $M_{\alpha, \Delta}$.
For $k>1$, the summand $L_k$ is a trivial $\Vir$-module.
Hence, by Lemma~\ref{lem:Triv*Irred=0} we have $[L_k{}_{(\lambda )} M_0] = 0$.

Finally, note that $[R_{(\lambda) } M_0]\subseteq R^{m+1}[\lambda ]$.
By the choice of $m$, $R^{m+1}$ is a trivial $L_1$-module. Therefore,
we may apply Lemma \ref{lem:Any*Irred=Triv} to conclude
$[R_{\lambda }M_0]=0$.

We have found an ideal $I=M_0\ne 0$ in $L$ which has zero intersection with
the center $Z$ of $L$ since $[v_{(0)}\cdot ]$ has no kernel in $M_0$.
\end{proof}

Now, let us expand the results of Proposition
\ref{prop:IdealExists}
to the more general case.

The conformal version of the Lie theorem
for solvable Lie conformal algebras \cite{DK1998}
implies, in particular,
that $R'=R^2$ is nilpotent.

\begin{prop}\label{prop:IdealExists2}
If $L_0$ contains a summand $L_1$ isomorphic either to $\Vir $ or
to $\Vir\ltimes \Cur\mathfrak g$ ($\mathfrak g$ is a
simple finite-dimensional Lie algebra)
such that $R$ is a nontrivial conformal $L_1$-module
then $L$ contains a nonzero ideal $I$ such that $I\cap Z(L)=0$.
\end{prop}

\begin{proof}

{\sc Case 1.} Assume $R'$ is a trivial $L_1$-module.
By Lemma \ref{lem:ExistIrreducible}, $R$ contains an
irreducible conformal $L_1$-submodule $M_0=M_{\alpha, \Delta, U}$.
For each $0\ne u\in U$, $Hu$ is a conformal module over $\Vir $
isomorphic to $M_{\alpha, \Delta}$.
Then $[L_i {}_{(\lambda )} Hu] = 0$
by Lemma \ref{lem:Triv*Irred=0}. Moreover,
$[R_{(\lambda )} Hu] \subseteq R'[\lambda ]$,
so we may apply Lemma \ref{lem:Any*Irred=Triv}
to conclude $[R_{\lambda } M_0] =0$.
Hence, $I=M_0$ is an ideal of $L$ which has no
intersection with $Z(L)$ since
$\mathrm{Ker}\,[v_{(0)}\cdot ]\cap M_0 = 0$.

{\sc Case 2.}
Assume $R'$ is a nontrivial $L_1$-module.
Then by Proposition \ref{prop:IdealExists}
there exists an ideal $M_0$ of $L_0\ltimes R'$
such that $[R'_{(\lambda )} M_0] = 0$.
As an $L_1$ module, $M_0$ is isomorphic to $M_{\alpha,\Delta,U}$,
where $U$ is either 1-dimensional or an irreducible
$\mathfrak g$-module.
We are going to prove that the ideal generated by $M_0$
in the entire algebra $L$  has trivial intersection with its center.

Consider $A = R/R'$ as a conformal module over $\Vir=Hv \subseteq L_1$
with respect to the induced regular action.
By Lemma \ref{lem:Triang}, there exists a sequence of conformal
$\Vir$-modules
$0=A_{-1}\subset A_0\subset \dots \subset A_n = A$,
where $A_k/A_{k-1}$ is either isomorphic to $M_{\alpha_k, \Delta_k}$
or trivial. Define an index $I(k)$, $k=0,\dots, n$, in the following way: $I(k)=1$ if $A_k/A_{k-1}\simeq M_{\alpha_k, \Delta_k}$
and $I(k)=0$ if $A_k/A_{k-1}$ is trivial (either torsion-free or coincides with its torsion).

Suppose $\bar a_k$, $k=0,\dots, n$, are the generators of
$A_k/A_{k-1}$ over $H$, and choose the corresponding pre-images $a_k\in R$.
Then the set $\{a_k+R' \mid k=0,\dots , n\}$ generates $R/R'$ over $H$,  and
\[
 D_0 a_k \in I(k)(\partial +\alpha _k)a_k + \sum\limits_{0\le i<k} f_{ik}(\partial )a_i + R',
\]
where $D_0 = [v_{(0)} \cdot]$
(we do not define what is $\alpha _k$ when $I(k)=0$).

Denote
\begin{multline}\label{eq:PrincipalWords}
w_{\alpha_{1,0},\dots, \alpha_{k_0,0},\dots ,\alpha_{1,n},\dots ,\alpha_{k_n,n}}^{k_0,\dots, k_n} \\
 =
 [a_n{}_{(\alpha_{1,n})}\dots [a_n{}_{(\alpha_{k_n,n})}
 [a_{n-1}{}_{(\alpha_{1,n-1})}
\dots
[a_0{}_{(\alpha_{1,0})}\dots  [a_0{}_{(\alpha_{k_0,0})} u]\dots ]]]],
\end{multline}
where $u\in U$, $\alpha_{l,i}\in \Bbbk $, $k_i\ge 0$.
Let $W(a_0,\dots, a_n)\subset R$ be the set of all
$w=w_{\alpha_{1,0},\dots, \alpha_{k_0,0},\dots ,\alpha_{1,n},\dots ,\alpha_{k_n,n}}^{k_0,\dots, k_n}$.
It is clear that the $H$-linear span $I$
of all elements from $W(a_0,\dots, a_n)$
is the ideal in $L$ generated by $M_0$.
Indeed, since $R/R'$ is an Abelian Lie conformal algebra and $[R'_{(\lambda )} u] =0$,
we have $[a_k{}_{(\alpha)} w]\in I$ for all $w\in I$ (one may re-arrange
the operators $[a_i{}_{(\alpha_{l,i})} \cdot ]$ in the desired way).
Obviously, $I$ is closed under the multiplication with $L_i$ for $i=1,\dots, s$,
but, in general, $[L_i{}_{(\lambda )} I]\ne 0$ in contrast to Proposition \ref{prop:IdealExists}.

Define the {\em weight\/} $\mathrm{wt}\,w $ of an expression $w$
of the form \eqref{eq:PrincipalWords}
as the $(n+2)$-tuple $(k_0+\dots +k_n, k_n,\dots, k_0)$,
and let the weights be the lexicographically ordered.

 Assume $Z(L)\cap I \ne 0$.
For every $z\in Z(L)\cap I$, $z\ne 0$,
there exists its presentation
\begin{equation}\label{eq:q-NormalForm}
 z=h_1(\partial )w_1 + \dots + h_m(\partial )w_m,\quad h_i\ne 0, \quad w_m\in W(a_0,\dots, a_n)  ;
 \end{equation}
such that $\max_i \mathrm{wt}\,w_i$ is minimal among all presentations of $z$ in the form
\eqref{eq:q-NormalForm}.
Denote such a weight by $\mathrm{wt}\,z$.
Then, consider those presentations of $z$ in the form \eqref{eq:q-NormalForm}
with $\max_i \mathrm{wt}\,w_i = \mathrm{wt}\,z$ and choose one with
minimal number $r$ of $w_i$s with $\mathrm{wt}\,w_i = \mathrm{wt}\,z$
(say, $\mathrm{wt}\,w_1 = \dots = \mathrm{wt}\,w_r = \mathrm{wt}\,z$,
$\mathrm{wt}\,w_i< \mathrm{wt}\,z $ for $i=r+1, \dots , m$).
Denote this number $r$ by $\deg z$.
Both
$\mathrm{wt}\,z$ and $\deg z$ are well-defined: They depend only in $z\in Z(L)\cap I$, $z\ne 0$.

Straightforward computation shows
\begin{equation}\label{eq:DerWAction}
D_0  w = (\partial + \alpha(w)) w+ z', \quad \alpha(w)\in \Bbbk ,\ \mathrm{wt}\,z'<\mathrm{wt}\, w
\end{equation}
for $w \in W(a_0,\dots, a_n)$.

We may choose $0\ne z_0\in Z(L)\cap I$ such that:
(1) $\mathrm{wt}\,z_0$ is minimal among all $z\in Z(L)\cap I$, $z\ne 0$;
(2) $r=\deg z_0$ is minimal possible among all $z$ with minimal $\mathrm{wt}\,z$.

Then
\[
 z_0 = h_1(\partial )w_1 + \dots + h_r(\partial )w_r + z',
\]
where $\mathrm{wt}\,w_1 = \dots = \mathrm{wt}\,w_r = \mathrm{wt}\,z_0$
and $\mathrm{wt}\,z' < \mathrm{wt}\,z_0 $.
But \eqref{eq:DerWAction} implies
\[
0 = D_0 z_0 = h_1(\partial )(\partial+\alpha(w_1))w_1 + \dots + h_r(\partial )(\partial+\alpha(w_r))w_r + z'',
\]
where $\mathrm{wt}\,z'' < \mathrm{wt}\,z_0$.
Hence,
\[
(\partial+\alpha(w_1))z_0\in (Z(L)\cap I)\setminus \{0\}
\]
has either smaller weight or smaller degree than $z_0$.
The contradiction obtained proves $I\cap Z(L)=0$.
\end{proof}

\begin{thm}
 Let $L$ be a finite torsion-free conformal Lie algebra
with a splitting solvable radical $R$.
Then $L$ has a finite faithful conformal representation.
\end{thm}

\begin{proof}
Suppose $L$ is a counterexample of minimal rank over $H$.
Then $Z(L)\ne 0$: Otherwise, the regular representation is faithful.

By Theorem  \ref{thm:ConfAlgebras}(ii),
$L_0 = L_1\oplus \dots \oplus L_k \oplus L_{k+1} \oplus \dots \oplus L_s$,
where $L_i$
are either $\Cur\mathfrak g$ or $\Vir $ or
$\Vir\ltimes \Cur\mathfrak g$.

Let $L_1,\dots , L_k$ be as in Lemma~\ref{lem:SplitExtMod}
(contain Virasoro element),
and let $L_{k+1}, \dots, L_s$ be isomorphic to current conformal algebras.
The case $k=0$ was considered in \cite{Kol2011}, so assume $k\ge 1$,
i.e., $L_0$ contains Virasoro elements.

The radical $R$ is a conformal module over $L_i$ for every $i=1,\dots, s$.
If $R$ is trivial over all $L_1,\dots, L_k$ then
$L$ can be presented as $L'\oplus L''$, where
\[
 L' = L_1\oplus \dots \oplus L_k,
\quad L'' = (L_{k+1}\oplus \dots \oplus L_s)\ltimes R.
\]
Here $L'$ is semisimple, $L''$ is of the kind considered
in \cite{Kol2011}.
Since both $L'$ and $L''$ have finite faithful conformal
representations, so is~$L$.
Therefore, we may assume $R$ is a nontrivial $L_1$-module.

Propositions \ref{prop:IdealExists} and   \ref{prop:IdealExists2}
imply the existence
of an ideal $I$ of $L$ such that
$I\cap Z(L)=0$.
Consider the set
$\hat I =\{a\in L\mid h(\partial ) a\in I \ \mbox{for some}\ h\in H\}\supseteq I$.
Since $[L_{(\lambda )} \hat I] \subseteq I[\lambda ]$, this is also an ideal in $L$.
Moreover,
if $I\cap Z(L)=0$ then $\hat I\cap Z(L) = 0$ (recall that the algebra $L$
is torsion-free). Both $L/Z(L)$ and $L/\hat I$ are torsion-free finite Lie conformal algebras of
smaller rank than $L$, thus have finite faithful representations.
The direct sum of these representations would be a faithful finite representation of~$L$.
\end{proof}

\section*{Acknowledgements}
I am grateful to Ivan Shestakov for communicating the
reference \cite{Ner2003}, and to the IME USP (Brazil)
where the major part of this work has been performed.

\end{document}